\documentclass{article}

\usepackage{amsmath}
\usepackage{amssymb}
\usepackage{amsfonts}
\usepackage{amsthm}
\usepackage{color}
\usepackage{graphicx}
\usepackage{hyperref}
\usepackage{mathabx}%%\notdivides
\usepackage{tikz}
\usetikzlibrary{calc}
\theoremstyle{plain}%plain,definition,remark. need amsthm package
\newtheorem{thm}{Theorem}[section]
\newtheorem{lem}[thm]{Lemma}
\newtheorem{prop}[thm]{Proposition}

\newtheorem{cor}[thm]{Corollory}
\theoremstyle{definition}
\newtheorem*{defn}{Definition}
\theoremstyle{remark}
\newtheorem{remark}[thm]{Remark}

\def\mod{\mathrm{mod}\ }
\def\gcd{\mathrm{gcd}}
\def\lcm{\mathrm{lcm}}
\def\PRIMES{\mathrm{PRIMES}}
\makeatletter % `@' now normal "letter"
\@addtoreset{equation}{section}
\makeatother  % `@' is restored as "non-letter"

\begin{document}

\title {\bf The largest cycles consist by the quadratic residues and Fermat primes}
\author{\it Haifeng Xu\thanks{Project supported by NSFC(Grant No. 11401515), the University Science Research Project of Jiangsu Province (14KJB110027) and the Foundation of Yangzhou University 2014CXJ004.}\\ {\small Yangzhou University}}
\date{\small\today}

\maketitle

\begin{abstract}
This paper studies the largest cycles consisted by the quadratic residues modulo prime numbers. We give some formulae about the maximum length of the cycles. Especially, the formula for modulo Fermat primes is given.
\end{abstract}

\noindent{\bf MSC2010:} 11A07.\\
{\bf Keywords:} quadratic residues, Fermat primes, largest cycles.

%\tableofcontents

%------------------------------
% Section 0. Introduction

%------------------------------
% Section 1.
\section{Examples and Definition}

We find this phenomenon by computation. So let us begin with some examples. Let's consider quadratic residues module $999$ first.

\[
\begin{cases}
454^2\equiv 322\ (\mod 999),\\
322^2\equiv 787\ (\mod 999),\\
787^2\equiv 988\ (\mod 999),\\
988^2\equiv 121\ (\mod 999),\\
121^2\equiv 655\ (\mod 999),\\
655^2\equiv 454\ (\mod 999),\\
\end{cases}
\qquad
\begin{cases}
445^2\equiv 223\ (\mod 999),\\
223^2\equiv 778\ (\mod 999),\\
778^2\equiv 889\ (\mod 999),\\
889^2\equiv 112\ (\mod 999),\\
112^2\equiv 556\ (\mod 999),\\
556^2\equiv 445\ (\mod 999).\\
\end{cases}
\]
We find that these numbers $\{454,322,787,988,121,655\}$ consist a cycle, also for the numbers $\{445,223,778,889,112,556\}$. The length of the cycle is $6$. In fact it is the maximum number of the elements in the cycles. We called it the length of the largest cycles for the quadratic residues of $999$. Look another example for modulus $99$, we get
\[
\begin{cases}
22^2\equiv 88\ (\mod 99),\\
88^2\equiv 22\ (\mod 99),\\
\end{cases}
\qquad
\begin{cases}
70^2\equiv 49\ (\mod 99),\\
49^2\equiv 25\ (\mod 99),\\
25^2\equiv 31\ (\mod 99),\\
31^2\equiv 70\ (\mod 99).\\
\end{cases}
\]
We see that, there exist a smaller cycle. We are interesting in the largest cycles.

\begin{defn}
Consider the equation $x^2\equiv a\ (\mod m)$. If there exist a series of numbers $\{x_i\}_{i=1}^{k}$ such that
\[
\begin{cases}
x_1^2 &\equiv x_2\ (\mod m),\\
x_2^2 &\equiv x_3\ (\mod m),\\
&\cdots\\
x_{k-1}^2 &\equiv x_k\ (\mod m),\\
x_k^2 &\equiv x_1\ (\mod m),\\
\end{cases}
\]
then we call these $k$ numbers consist a cycle modulo $m$. The number $k$ is defined as the length of the cycle.
\end{defn}

It infers that
\[
x_i^{2^k}\equiv x_i\ (\mod m),\quad\text{for}\ i=1,2,\ldots,k.
\]
There exists at least one largest cycle. We denote the length of it by $L(m)$. For example, $L(99)=4$, $L(999)=6$. In the next section, we try to find out the formula for $L(m)$.

%------------------------------

\section{Computations}

The maximum length of the cycles modulo prime numbers are list as follows:

\medskip

\begin{tabular}{|c|r|r|r|r|r|r|r|r|r|r|}
\hline
$m=p$ & 2 & 3 & 5 & 7 & 11 & 13 & 17 & 19 & 23 & 29 \\
\hline
$L(p)$ & 1 & 1 & 1 & 2 &  4 &  2 &  1 &  6 & 10 &  3 \\
\hline
\hline
$m=p$ & 31 & 37 & 41 & 43 & 47 & 53 & 59 & 61 & 67 & 71 \\
\hline
$L(p)$ & 4 &  6 & 4 & 6 & 11 & 12 &  28 &  4 &  10 &  12 \\
\hline
\hline
$m=p$ & 73 & 79 & 83 & 89 & 97 & 101 & 103 & 107 & 109 & 113 \\
\hline
$L(p)$ & 6 &  12 &  20 &  10 & 2 & 20 & 8 & 52 &  18 &  3 \\
\hline
\hline
$m=p$ & 127 & 131 & 137 & 139 & 149 & 151 & 157 & 163 & 167 & $\ldots$ \\
\hline
$L(p)$ &  6 &  12 & 8 & 22 &  36 & 20 & 12 & 54 & 82 & $\ldots$ \\
\hline
\end{tabular}

\bigskip

Obviously, it establishes a map from the set of primes $\PRIMES$ to the set of positive numbers $\mathbb{Z}^{+}$:
\[
\begin{array}{rcl}
L:\ \PRIMES &\rightarrow &\mathbb{Z}^{+},\\
    p & \mapsto & L(p).\\
\end{array}
\]

\bigskip

In fact, E. L. Blanton, Jr., S. P. Hurd and J. S. McCranie \cite{Blanton-Hurd-McCranie} had considered this question. We can find this sequence on the site of OEIS \cite{OEIS-A037178}.

For composites, we have

\medskip

\begin{tabular}{|c|r|r|r|r|r|r|r|r|r|r|r|r|}
\hline
$m=c$ & 4 & 6 & 8 & 9 & 10 & 12 & 14 & 15 & 16 & 18 & 20 & 21 \\
\hline
$L(c)$ & 1 & 1 & 1 & 2 & 1 & 1 & 2 & 1 & 1 & 2 & 1 & 2 \\
\hline
\hline
$m=c$ & 22 & 24 & 25 & 26 & 27 & 28 & 30 & 32 & 33 & 34 & 35 & 36 \\
\hline
$L(c)$ & 4 & 1 & 4 & 2 & 6 & 2 & 1 & 1 & 4 & 1 & 2 & 2 \\
\hline
\hline
$m=c$ & 38 & 39 & 40 & 42 & 44 & 45 & 46 & 48 & 49 & 50 & 51 & $\ldots$ \\
\hline
$L(c)$ & 6 & 2 & 1 & 2 & 4 & 2 &  10 & 1 & 6 & 4 &  1 & $\ldots$ \\
\hline
\end{tabular}

\bigskip

And when modulo $p^2$, we have

\medskip

\begin{tabular}{|c|r|r|r|r|r|r|r|r|r|r|}
\hline
$m=p^2$ & $2^2$ & $3^2$ & $5^2$ & $7^2$ & $11^2$ & $13^2$ & $17^2$ & $19^2$ & $23^2$ & $29^2$ \\
\hline
$L(p^2)$ & 1 & 2 & 4 & 6 & 20 & 12 & 8 & 18 & 110 & 84 \\
\hline
\hline
$m=p^2$ & $31^2$ & $37^2$ & $41^2$ & $43^2$ & $47^2$ & $53^2$ & $59^2$ & $61^2$ & $67^2$ & $71^2$ \\
\hline
$L(p^2)$ & 20 & 36 & 20 & 42 & 253 & 156 & 812 & 60 & 330 & 420 \\
\hline
\hline
$m=p^2$ & $73^2$ & $79^2$ & $83^2$ & $89^2$ & $97^2$ & $101^2$ & $103^2$ & $107^2$ & $109^2$ & $113^2$ \\
\hline
$L(p^2)$ & 18 & 156 & 820 & 110 & 48 & 100 & 408 & 2756 & 36 & 84 \\
\hline
\hline
$m=p^2$ & $127^2$ & $131^2$ & $137^2$ & $139^2$ & $149^2$ & $151^2$ & $157^2$ & $163^2$ & $167^2$ & $\ldots$ \\
\hline
$L(p^2)$ & 42 & 780 & 136 & 1518 & 1332 & 60 & 156 & 162 & 6806 & $\ldots$ \\
\hline
\end{tabular}

\bigskip

And consider modulo the powers of prime numbers. For examples:
\medskip

\begin{tabular}{|c|r|r|r|r|r|r|}
\hline
$m=p^n$  & 4 & 8  & 16 & 32  & 64 &  $\ldots$ \\
\hline
$L(p^n)$ & 1 & 1  & 1  & 1   & 1   & $\ldots$  \\
\hline
\hline
$m=p^n$ & $3^2$ & $3^3$ & $3^4$ & $3^5$ & $3^6$  & $\ldots$ \\
\hline
$L(p^n)$ & 2 & $2\times 3$ & $3\times 6$ & $6\times 9$ & $9\times 18$ & $\ldots$ \\
\hline
\hline
$m=p^n$ & $5^2$ & $5^3$ & $5^4$ & $5^5$ & $5^6$  & $\ldots$ \\
\hline
$L(p^n)$ & 4 & $4\times 5$ & $5\times 20$ & $20\times 25$ & $25\times 100$ & $\ldots$ \\
\hline
\hline
$m=p^n$ & $7^2$ & $7^3$ & $7^4$ & $7^5$ & $7^6$  & $\ldots$ \\
\hline
$L(p^n)$ & 6 & $6\times 7$ & $7\times 42$ & $42\times 49$ & $\ldots$ & $\ldots$ \\
\hline
\hline
$m=p^n$ & $11^2$ & $11^3$ & $11^4$ & $11^5$ & $11^6$  & $\ldots$ \\
\hline
$L(p^n)$ & $20$ & $20\times 11$ & $220\times 11$ & $26620$ & $\ldots$ & $\ldots$ \\
\hline
\hline
$m=p^n$ & $13^2$ & $13^3$ & $13^4$ & $13^5$ & $13^6$  & $\ldots$ \\
\hline
$L(p^n)$ & $12$ & $12\times 13$ & $13\times 156$ & $\ldots$ & $\ldots$ & $\ldots$ \\
\hline
\hline
$m=p^n$ & $17^2$ & $17^3$ & $17^4$ & $17^5$ & $17^6$  & $\ldots$ \\
\hline
$L(p^n)$ & $8$ & $8\times 17$ & $17\times 136$ & $\ldots$ & $\ldots$ & $\ldots$ \\
\hline
\hline
$m=p^n$ & $19^2$ & $19^3$ & $19^4$ & $19^5$ & $19^6$  & $\ldots$ \\
\hline
$L(p^n)$ & $18$ & $18\times 19$ & $342\times 19$ & $\ldots$ & $\ldots$ & $\ldots$ \\
\hline
\hline
$m=p^n$ & $23^2$ & $23^3$ & $23^4$ & $23^5$ & $23^6$  & $\ldots$ \\
\hline
$L(p^n)$ & $110$ & $110\times 23$ & $\ldots$ & $\ldots$ & $\ldots$ & $\ldots$ \\
\hline
\hline
$m=p^n$ & $29^2$ & $29^3$ & $29^4$ & $29^5$ & $29^6$  & $\ldots$ \\
\hline
$L(p^n)$ & $84$ & $84\times 29$ & $\ldots$ & $\ldots$ & $\ldots$ & $\ldots$ \\
\hline
\hline
$m=p^n$ & $31^2$ & $31^3$ & $31^4$ & $31^5$ & $31^6$  & $\ldots$ \\
\hline
$L(p^n)$ & $20$ & $20\times 31$ & $\ldots$ & $\ldots$ & $\ldots$ & $\ldots$ \\
\hline
\hline
$m=p^n$ & $47^2$ & $47^3$ & $47^4$ & $47^5$ & $47^6$  & $\ldots$ \\
\hline
$L(p^n)$ & $253$ & $253\times 47$ & $\ldots$ & $\ldots$ & $\ldots$ & $\ldots$ \\
\hline
%\hline
%$m=p^n$ & $53^2$ & $53^3$ & $23^4$ & $23^5$ & $23^6$  & $\ldots$ \\
%\hline
%$L(p^n)$ & $156$ & $\ldots$ & $\ldots$ & $\ldots$ & $\ldots$ & $\ldots$ \\
%\hline
%\hline
%$m=p^n$ & $59^2$ & $59^3$ & $23^4$ & $23^5$ & $23^6$  & $\ldots$ \\
%\hline
%$L(p^n)$ & $812$ & $\ldots$ & $\ldots$ & $\ldots$ & $\ldots$ & $\ldots$ \\
%\hline
%\hline
%$m=p^n$ & $71^2$ & $71^3$ & $23^4$ & $23^5$ & $23^6$  & $\ldots$ \\
%\hline
%$L(p^n)$ & $420$ & $\ldots$ & $\ldots$ & $\ldots$ & $\ldots$ & $\ldots$ \\
%\hline
%\hline
%$m=p^n$ & $83^2$ & $83^3$ & $23^4$ & $23^5$ & $23^6$  & $\ldots$ \\
%\hline
%$L(p^n)$ & $820$ & $\ldots$ & $\ldots$ & $\ldots$ & $\ldots$ & $\ldots$ \\
%\hline
%\hline
%$m=p^n$ & $89^2$ & $89^3$ & $23^4$ & $23^5$ & $23^6$  & $\ldots$ \\
%\hline
%$L(p^n)$ & $110$ & $\ldots$ & $\ldots$ & $\ldots$ & $\ldots$ & $\ldots$ \\
%\hline
\end{tabular}

\bigskip

The data listed in these tables are all verified by computations.

$L(p^2)$ is relevant to $p-1$ and $L(p)$. It is rather complicated and we discuss it later. But for few exceptions we have the explicit formula. For example, the Fermat primes (Proposition \ref{prop:1} and \ref{prop:2}).

For the first five Fermat primes $F_0=3$, $F_1=5$, $F_2=17$, $F_3=257$, $F_4=65537$, we have $L(F_i)=1$, $i=0,1,2,3,4$.

For $p^2=11^2$, we have $L(121)=20$. The largest cycle $\{x_1,x_2,\ldots,x_{20}\}$ is
\[
4, 16, 14, 75, 59, 93, 58, 97, 92, 115, 36, 86, 15, 104, 47, 31, 114, 49, 102, 119.
\]
Of course, there are also cycles with length equal to $10$. For example, one of them is
\[
100, 78, 34, 67, 12, 23, 45, 89, 56, 111.
\]

For $p^2=17^2$, we have $L(289)=8=\frac{1}{2}(17-1)$. For example, one of the cycles is
\[
256, 222, 154, 18, 35, 69, 137, 273.
\]

For $p^2=257^2$, we have $L(66049)=16=\frac{1}{2^4}(257-1)$. For example, one of cycles is
\[
\begin{array}{rrrrrrrr}
65536, & 65022, & 63994, & 61938, & 57826, & 49602, & 33154, & 258,\\
 515,  & 1029,  & 2057,  & 4113,  & 8225,  & 16449, & 32897, & 65793.\\
\end{array}
\]

For $p^2=65537^2$, we have $L(4295098369)=32=\frac{1}{2^{11}}(65537-1)$. For example, one of the cycles is
\[
\begin{array}{rrrr}
4294967296, & 4294836222, & 4294574074, & 4294049778, \\
4293001186, & 4290904002, & 4286709634, & 4278320898, \\
4261543426, & 4227988482, & 4160878594, & 4026658818, \\
3758219266, & 3221340162, & 2147581954, & 65538, \\
131075, & 262149, & 524297, & 1048593, \\
2097185, & 4194369, & 8388737, & 16777473, \\
33554945, & 67109889, & 134219777, & 268439553, \\
536879105, & 1073758209, & 2147516417, & 4295032833 \\
\end{array}
\]

Note that
\[
17=2^{2^2}+1, \quad 257=2^{2^3}+1,\quad 65537=2^{2^4}+1.
\]
They are the Fermat numbers. For $k=0,1,2,3,4$, $2^{2^k}+1=3,5,17,257,65537$ are all prime numbers. But $2^{32}+1=4294967297$ is not a prime. It equals $641 \times 6700417$. It is conjectured that there are only 5 terms. Currently it has been shown that $2^{2^k} + 1$ is composite for $5\leqslant k\leqslant 32$ \cite{FermatPrime}, \cite{OEIS-A019434}.

%------------------------------
\section{Some lemmas}

\begin{defn}[\cite{bakkerbe}]
Let $n\in\mathbb{N}$. A primitive root mod $n$ is a residue class $\alpha\in(\mathbb{Z}/n\mathbb{Z})^*$ with maximal order, i.e., $\mathrm{ord}(\alpha)=\varphi(n)$.
\end{defn}

\begin{lem}\label{lem:Euler-totient-function}
The Euler's totient function $\varphi(m)$ have the following formula for $m=p^n$:
\[
\varphi(p^n)=(p-1)p^{n-1}.
\]
\end{lem}

Euler's totient function is a multiplicative function, meaning that if two numbers $m$ and $n$ are coprime, then $\varphi(mn) = \varphi(m) \varphi(n)$.

\begin{lem}[Euler's Criterion]
Let $p$ be an odd prime and $a$ not divisible by $p$. Then $a$ is a quadratic residue modulo $p$ if and only if
\[
a^{(p-1)/2}\equiv 1\ (\mod p).
\]
\end{lem}

The proof uses the fact that the residue classes modulo a prime number are a field.

\begin{lem}[\cite{K-L-S_2001}, \cite{Tsang}]\label{lem:primitive-root}
The set of all quadratic nonresidues of a Fermat prime is equal to the set of all its primitive roots.
\end{lem}
\begin{proof}
Let $a$ be a quadratic nonresidue of the Fermat prime $F_n$, and let $e=\mathrm{ord}_{F_n}a$. According to Fermat's little theorem, $a^{F_n-1}\equiv 1\ (\mod F_n)$. So $e | F_n-1=2^{2^n}$. It follows that $e=2^k$ for some nonnegative $k\leqslant 2^n$. On the other hand, by Euler's criterion,
\[
a^{(F_n-1)/2}=2^{2^{2^n-1}}\equiv -1\ (\mod F_n).
\]
Hence, if $k < 2^n$, then $2^k | 2^{2^n-1}$ and so $a^{2^{2^n-1}}\equiv 1\ (\mod F_n)$, which is a contradiction. So, $k=2^n$ and $\mathrm{ord}_{F_n}a=2^{2^n}$. Therefore, $a$ is a primitive root modulo $F_n$.
\end{proof}

\begin{lem}[\cite{bakkerbe}]\label{lem:bakkerbe}
Let $n\in\mathbb{N}$, $a\in\mathbb{Z}$ coprime to $n$. Consider the equation
\[
x^d\equiv a\ (\mod n)
\]
$\mathrm{(a)}$ It has a solution if and only if $a^{\varphi(n)/\mathrm{gcd}(d,\varphi(n))}\equiv 1\ (\mod n)$.\\
$\mathrm{(b)}$ If it has a solution, it has exactly $\mathrm{gcd}(d,\varphi(n))$ solutions modulo $n$.\\
$\mathrm{(c)}$ If $x$ is a solution to this equation, then any other solution $x'$ satisfies $x'\equiv yx\ (\mod n)$ for a unique solution $y\ \mod n$ to $y^d\equiv 1\ (\mod n)$.
\end{lem}

On one extreme, this means if $\gcd(d,\varphi(n))=1$, then there is a unique $d$-th root $\mod n$ of any $a\in\mathbb{Z}$ coprime to $n$. On the other hand, it means that for any $d|\varphi(n)$, there are exactly $d$ solutions $\mod n$ to the equation $x^d\equiv 1\ (\mod n)$. This in fact characterizes the existence of primitive roots:

\begin{lem}\label{lem:bakkerbe2}
Let $n\in\mathbb{N}$. The following are equivalent:\\
$\mathrm{(a)}$ There is a primitive root $\mod n$.\\
$\mathrm{(b)}$ For all $d\in\mathbb{N}$ with $d|\varphi(n)$, there are exactly $d$ solutions $\mod n$ to the equation $x^d\equiv 1\ (\mod n)$.\\
$\mathrm{(c)}$ For all $d\in\mathbb{N}$ with $d|\varphi(n)$, there are exactly $\varphi(d)$ elements of $(\mathbb{Z}/n\mathbb{Z})^*$ of order $d$.
\end{lem}

\begin{cor}\label{cor:1}
Given any odd prime number $p$. Let $d=\mathrm{gcd}\bigl(2^{p-1}-1,(p-1)p\bigr)$. Then there are exactly $d$ solutions mod $p^2$ to the equation
\[
x^d\equiv 1\ (\mod p^2).
\]
\end{cor}

%---------------------------------
\section{Main results}

\begin{prop}\label{prop:1}
For the Fermat primes, i.e., the prime numbers of the form $p=2^{2^k}+1$, we have $L(p)=1$.
\end{prop}
\begin{proof}
Suppose $F_k=2^{2^k}+1$ is a prime, $k\geqslant 1$. Then $\varphi(F_k)=2^{2^k}$. To find the largest cycle of the quadratic residue equation of modulo $F_k$. We consider the equation
\[
x^{2^m}\equiv x\ (\mod F_k),
\]
where $x$ is coprime with $F_k$. Then it is equivalent to the equation
\begin{equation}\label{eqn:1}
x^{2^m-1}\equiv 1\ (\mod F_k).
\end{equation}
By Lemma \ref{lem:bakkerbe}, \eqref{eqn:1} have a solution. And the number of the solutions is equal to
\[
\gcd(2^m-1, 2^{2^k})=1.
\]
Since $x=1$ is the trivial solution of \eqref{eqn:1}, hence there is no other solutions. Hence we have $L(F_k)=1$.
\end{proof}

%-----------------------------
% this remark can be deleted
\begin{remark}
Since there is a primitive root $\mod F_k$, by Lemma \ref{lem:bakkerbe2}, for any $1\leqslant m\leqslant 2^k$, $x^{2^m}\equiv 1\ (\mod F_k)$ have exact $2^m$ solutions.
Suppose $k\geqslant 1$. Let's consider the equation
\[
x^2\equiv -1\ (\mod F_k).
\]
By Lemma \ref{lem:bakkerbe}, it has a solution since $(-1)^{2^{2^k-1}}\equiv 1\ (\mod F_k)$. And there are exactly two solutions of this equation. By Lemma \ref{lem:primitive-root}, these two solutions $\{x_1,x_2\}$ are the primitive roots modulo $F_k$. By definition of primitive root, they have the maximal order. That is, $\mathrm{ord}(x_1)=\mathrm{ord}(x_2)=\varphi(F_k)=2^{2^k}$. Hence, any other number $a$ coprime to $F_k$ (not the primitive root) with order $\mathrm{ord}(a)<2^{2^k}$. And $\mathrm{ord}(a)| 2^{2^k-1}$. Hence, we have
\begin{equation}\label{eqn:0}
a^{2^{2^k-1}}\equiv 1\ (\mod F_k).
\end{equation}
It infers that the equation $x^2\equiv a\ (\mod F_k)$ has a solution for any $a$ coprime to $F_k$. (Use Lemma \ref{lem:bakkerbe} again.)
\end{remark}

\begin{prop}\label{prop:2}
For the Fermat primes, i.e., the prime numbers of the form $p=2^{2^k}+1$, we have
\[
L(p^2)=L((2^{2^k}+1)^2)=\dfrac{1}{2^{2^k-k-1}}(p-1)=2^{k+1}.
\]
\end{prop}
\begin{proof}
We have verified it for $k=0,1,2,3,4$, the known Fermat primes. Suppose $p=F_k=2^{2^k}+1$ is a prime, $k\geqslant 1$. Similar to the Proposition \ref{prop:1}, we consider the equation
\[
x^{2^m}\equiv x\ (\mod F_k^2),\quad x\neq F_k,
\]
and find the nontrivial solution for the maximal $m$. Since $F_k$ is a prime, for $x\neq F_k$, $x$ is coprime with $F_k^2$. Then it is equivalent to
\begin{equation}\label{eqn:2}
x^{2^m-1}\equiv 1\ (\mod F_k^2).
\end{equation}
By Lemma \ref{lem:bakkerbe}, there are $\gcd(2^m-1, \varphi(F_k^2))$ solutions. By Lemma \ref{lem:Euler-totient-function}, $\varphi(F_k^2)=(F_k-1)F_k=2^{2^k}(2^{2^k}+1)$. Thus, if $\gcd(2^m-1, 2^{2^k}+1)>1$, the equation \eqref{eqn:2} has nontrivial solutions. The number of the solutions is equal to $2^{2^k}+1$.

Because $F_k=2^{2^k}+1$ is a prime number,
\[
\gcd\bigl(2^m-1, 2^{2^k}+1\bigr)>1\ \Leftrightarrow\ (2^{2^k}+1)\,\bigr|\,(2^m-1).
\]
It infers that $m=h\cdot 2^{k+1}$, here $h\geqslant 1$. And there are $2^{2^k}+1$ solutions include the trivial $x=1$. So, for example, if $h=2$, then
\[
x^{2^m}=x^{2^{2^{k+1}}\cdot 2^{2^{k+1}}}=\Bigl(x^{2^{2^{k+1}}}\Bigr)^{2^{2^{k+1}}}\equiv x^{2^{2^{k+1}}}\equiv x\ (\mod F_k^2).
\]
Hence, $L(F_k^2)=2^{k+1}$ for Fermat prime $F_k$.
\end{proof}

%------------------------
% The following Theorem is found by Dr. Timothy

Since there are $(p-1)/2$ quadratic residues mod $p$, we have $L(p)\leq (p-1)/2$. However, $1$ maps to itself under squaring mod $p$, so we expect $L(p)\leq (p-3)/2$. In the theorem below, we state a condition for this to happen.

\begin{thm}\label{thm:by-Dr.Timothy}
Let $p$ be a non-Fermat prime. $L(p)=(p-3)/2$ if and only if $q=(p-1)/2$ is prime and $2$ is a generator of $(\mathbb{Z}/q\mathbb{Z})^{*}$.
\end{thm}
\begin{proof}
Let $d$ be an odd divisor of $p-1$. Then there is a solution to the equations
\begin{equation}\label{1}
x^{d}\equiv 1\ (\mod p),\quad x^i\not\equiv 1\ (\mod p),
\end{equation}
$1\leq i<d$. Let $x$ satisfy $\eqref{1}$.
Then such $x$ is a solution to the equation
$$
x^{2^m}\equiv x\ (\mod p)
$$
for all $m$ such that $\mbox{ord}_d(2)|m$ where $\mbox{ord}_d(2)$ is the multiplicative order of $2$ mod $d$. Furthermore,
$$
L(p)=\max_{\substack{d|p-1\\d \equiv 1 \bmod 2}}\mbox{ord}_d(2)\leq\max_{\substack{d|p-1\\d \equiv 1 \bmod 2}}d-1.
$$
We need only consider the case where $(p-1)/2$ is odd, since if $d<(p-1)/2$, then $\mbox{ord}_d(2)<(p-3)/2$. Now any odd $d$ dividing $(p-1)$ satisfies $\mbox{ord}_d(2)\leq \mbox{ord}_{(p-1)/2}(2)$, hence $L(p)=\mbox{ord}_{(p-1)/2}(2)$. This is $(p-3)/2$ if and only if $q=(p-1)/2$ is prime and $2$ is a generator of $(\mathbb{Z}/q\mathbb{Z})^{*}$.
\end{proof}
\begin{remark}\label{rem:by-Dr.Timothy}
Similar to the case for $L(p)$, we have
\[
L(p^2)=\max_{\substack{d|p(p-1)\\d \equiv 1 \bmod 2}}\mbox{ord}_d(2).
\]
Let $n$ be the largest odd divisor of $p-1$. Then $L(p)=\mbox{ord}_n(2)$ and $L(p^2)=\mbox{ord}_{np}(2)=\lcm(L(p),\mbox{ord}_p(2))$, where $\mathrm{lcm}(a_1,\ldots,a_n)$ stands for the lowest common multiple of $a_1,\ldots,a_n$.
\end{remark}
% the above theorem is from Dr. Timothy
%---------------

From the above tables, we guess there are formulas for $L(p^n)$ for $n\geqslant 2$.

\begin{prop}\label{prop:3}
Let $p$ be a prime. Then we have
\[
L(p^2)\leqslant(p-1)L(p).
\]
\end{prop}
\begin{proof}
Suppose $p$ is a prime. Similar to the Proposition \ref{prop:1}, we consider the equation
\[
x^{2^m}\equiv x\ (\mod p^2),
\]
and find the nontrivial solution for the maximal $m$. Suppose $x$ is coprime with $p$. Then it is equivalent to
\begin{equation}\label{eqn:3}
x^{2^m-1}\equiv 1\ (\mod p^2).
\end{equation}
By Lemma \ref{lem:bakkerbe}, there are $\gcd(2^m-1, \varphi(p^2))$ solutions. By Lemma \ref{lem:Euler-totient-function}, $\varphi(p^2)=(p-1)p$. By Euler theorem, for every odd prime $p$, we have
\[
2^{p-1}\equiv 1\ (\mod p).
\]
Thus, $p|(2^{p-1}-1)$ for $p>2$. Hence
\[
\begin{split}
& \gcd(2^m-1, (p-1)p)>1\\
\Leftrightarrow\ &\gcd(2^{p-1}-1,2^{m}-1)>1\quad\text{or}\quad\gcd(p-1,2^m-1)>1.\\
\end{split}
\]
Hence for find the largest $m$, we only need to consider $\gcd(2^{p-1}-1,2^{m}-1)>1$ with $m\geqslant p-1$.  Which infers that $m=h(p-1)$, $h\geqslant 1$.

On the other hand, we can proved that $h\leqslant L(p)$. In fact Figure \ref{Fig:mod-p2} describes the relationship between modulo $p$ and modulo $p^2$. Suppose
\[
x_1^2\equiv x_2\ (\mod p),\quad x_1^2\equiv x'_2\ (\mod p^2).
\]
Then there exists some integers $s$ and $t$ such that
\[
x_1^2=x_2+sp,\quad x_1^2=x'_2+tp^2.
\]
Thus, we have
\[
x_2-x'_2=p(ps-t).
\]
It shows that $x'_2$ and $x_2$ lie in the same vertical line in Figure \ref{Fig:mod-p2}. And ${x'}^2_2\equiv x_2^2\equiv x_3\ (\mod p)$. Let ${x'}^2_2\equiv x'_3\ (\mod p^2)$. Then it infers that $ x'_3\equiv x_3\ (\mod p)$. Hence, each $x'_i$ lies in the same line of $x_i$. Therefore, the elements in the biggest cycle all lie  in the area bounded by dashed lines below. Because there are no more than $p-1$ vertical lines, the total numbers of the elements in the biggest cycle is less or equal to $(p-1)L(p)$. It means that $L(p^2)\leqslant (p-1)L(p)$.

%In fact, if let $m=(p-1)L(p)$, then \eqref{eqn:3} has exactly
%\[
%\gcd(2^{(p-1)L(p)}-1,(p-1)p)
%\]
%solutions.

\begin{figure}[htbp]
  \centering
\begin{tikzpicture}[scale=0.8]%for scaling
\draw[gray] (0,0) -- (11,0);

\fill [blue] ($(0,0)$) circle (1pt);
\fill [black] ($(1,0)$) circle (1pt);
\fill [black] ($(2,0)$) circle (1pt);
\fill [black] ($(3,0)$) circle (1pt);
\fill [red] ($(4,0)$) circle (1.5pt);
\fill [black] ($(5,0)$) circle (1pt);
\fill [red] ($(6,0)$) circle (1.5pt);
\fill [black] ($(7,0)$) circle (1pt);
\fill [red] ($(8,0)$) circle (1.5pt);
\fill [red] ($(9,0)$) circle (1.5pt);
\fill [black] ($(10,0)$) circle (1pt);
\fill [blue] ($(11,0)$) circle (1.5pt);

\node  at (0,0)[anchor=north] {\footnotesize $(p-1)p$};
\node  at (1,0)[anchor=south] {\footnotesize $+1$};
\node  at (2,0)[anchor=south] {\footnotesize $+2$};
\node  at (10,0)[anchor=south] {\footnotesize $+p-1$};
\node  at (11,0)[anchor=south] {\footnotesize $p^2$};

%-------------------------------
\draw[gray] (0,1) -- (11,1);

\fill [blue] ($(0,1)$) circle (1pt);
\fill [black] ($(1,1)$) circle (1pt);
\fill [black] ($(2,1)$) circle (1pt);
\fill [black] ($(3,1)$) circle (1pt);
\fill [red] ($(4,1)$) circle (1.5pt);
\fill [black] ($(5,1)$) circle (1pt);
\fill [red] ($(6,1)$) circle (1.5pt);
\fill [black] ($(7,1)$) circle (1pt);
\fill [red] ($(8,1)$) circle (1.5pt);
\fill [red] ($(9,1)$) circle (1.5pt);
\fill [black] ($(10,1)$) circle (1pt);
\fill [blue] ($(11,1)$) circle (1.5pt);

\node  at (0,1)[anchor=north] {\footnotesize $(p-2)p$};
\node  at (1,1)[anchor=south] {\footnotesize $+1$};
\node  at (2,1)[anchor=south] {\footnotesize $+2$};
\node  at (10,1)[anchor=south] {\footnotesize $+p-1$};
\node  at (11.4,1)[anchor=south] {\footnotesize $(p-1)p$};

\node  at (6,1)[anchor=south] {\footnotesize $x'_{L(p)}$};

%-------------------------------
\draw[gray] (0,3) -- (11,3);

\fill [blue] ($(0,3)$) circle (1pt);
\fill [black] ($(1,3)$) circle (1pt);
\fill [black] ($(2,3)$) circle (1pt);
\fill [black] ($(3,3)$) circle (1pt);
\fill [red] ($(4,3)$) circle (1.5pt);
\fill [black] ($(5,3)$) circle (1pt);
\fill [red] ($(6,3)$) circle (1.5pt);
\fill [black] ($(7,3)$) circle (1pt);
\fill [red] ($(8,3)$) circle (1.5pt);
\fill [red] ($(9,3)$) circle (1.5pt);
\fill [black] ($(10,3)$) circle (1pt);
\fill [blue] ($(11,3)$) circle (1.5pt);

\node  at (0,3)[anchor=north] {\footnotesize $(k-1)p$};
\node  at (1,3)[anchor=south] {\footnotesize $+1$};
\node  at (2,3)[anchor=south] {\footnotesize $+2$};
\node  at (10,3)[anchor=south] {\footnotesize $+p-1$};
\node  at (11,3)[anchor=south] {\footnotesize $kp$};

\node  at (4,3)[anchor=south] {\footnotesize $x'_1$};
\node  at (9,3)[anchor=south] {\footnotesize $x'_i$};

%-------------------------------
\draw[gray] (0,5) -- (11,5);

\fill [blue] ($(0,5)$) circle (1pt);
\fill [black] ($(1,5)$) circle (1pt);
\fill [black] ($(2,5)$) circle (1pt);
\fill [black] ($(3,5)$) circle (1pt);
\fill [red] ($(4,5)$) circle (1.5pt);
\fill [black] ($(5,5)$) circle (1pt);
\fill [red] ($(6,5)$) circle (1.5pt);
\fill [black] ($(7,5)$) circle (1pt);
\fill [red] ($(8,5)$) circle (1.5pt);
\fill [red] ($(9,5)$) circle (1.5pt);
\fill [black] ($(10,5)$) circle (1pt);
\fill [blue] ($(11,5)$) circle (1.5pt);

\node  at (0,5)[anchor=north] {\footnotesize $p$};
\node  at (1,5)[anchor=south] {\footnotesize $+1$};
\node  at (2,5)[anchor=south] {\footnotesize $+2$};
\node  at (10,5)[anchor=south] {\footnotesize $+p-1$};
\node  at (11,5)[anchor=south] {\footnotesize $2p$};

\node  at (8,5)[anchor=south] {\footnotesize $x'_2$};

%-------------------------------
\draw[gray] (0,6) -- (11,6);

\fill [blue] ($(0,6)$) circle (1pt);
\fill [black] ($(1,6)$) circle (1pt);
\fill [black] ($(2,6)$) circle (1pt);
\fill [black] ($(3,6)$) circle (1pt);
\fill [red] ($(4,6)$) circle (1.5pt);
\fill [black] ($(5,6)$) circle (1pt);
\fill [red] ($(6,6)$) circle (1.5pt);
\fill [black] ($(7,6)$) circle (1pt);
\fill [red] ($(8,6)$) circle (1.5pt);
\fill [red] ($(9,6)$) circle (1.5pt);
\fill [black] ($(10,6)$) circle (1pt);
\fill [blue] ($(11,6)$) circle (1.5pt);

\node  at (0,6)[anchor=north] {\footnotesize $0$};
\node  at (1,6)[anchor=south] {\footnotesize $+1$};
\node  at (2,6)[anchor=south] {\footnotesize $+2$};
\node  at (10,6)[anchor=south] {\footnotesize $+p-1$};
\node  at (11,6)[anchor=south] {\footnotesize $p$};

\node  at (4,6)[anchor=south] {\footnotesize $x_1$};
\node  at (6,6)[anchor=south] {\footnotesize $x_{L(p)}$};
\node  at (8,6)[anchor=south] {\footnotesize $x_2$};
\node  at (9,6)[anchor=south] {\footnotesize $x_i$};

\draw[gray,dashed] (3.8,-0.5) -- (3.8,6.5);
\draw[gray,dashed] (9.2,-0.5) -- (9.2,6.5);
\draw[gray,dashed] (3.8,6.5) -- (9.2,6.5);
\draw[gray,dashed] (3.8,-0.5) -- (9.2,-0.5);
\node  at (6.5,4) {\footnotesize $\leqslant(p-1)L(p)$ elements};

%\draw[blue][->] (5.5,-0.8) -- (5.5,-0.1);
%\draw[blue][->] (6.5,-0.8) -- (6.5,-0.1);
%\draw[blue][->] (8.7,-0.8) -- (8.7,-0.1);
\end{tikzpicture}
  \caption{$\mod p$ and $\mod p^2$}\label{Fig:mod-p2}
\end{figure}
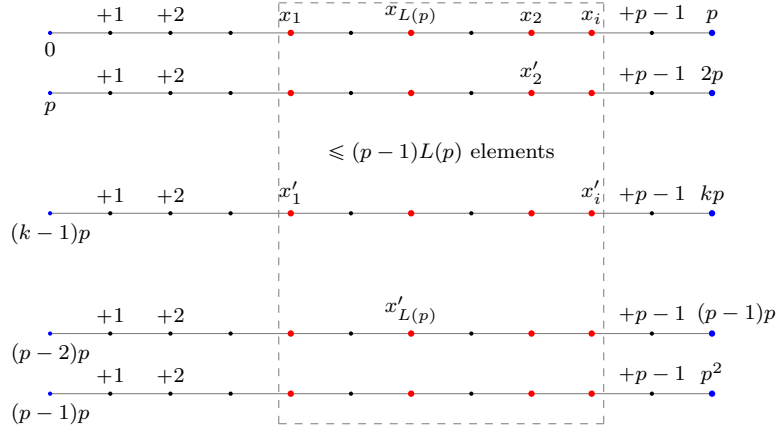

\end{proof}

\begin{prop}
For $n\geqslant 3$, we have
\[
L(p^n)=p^{n-2}\cdot L(p^2).
\]
\end{prop}
\begin{proof}
%Consider equation
%\[
%x^{2^m-1}\equiv 1\ (\mod p^n),
%\]
%where $(x,p^n)=1$. By Lemma \ref{lem:bakkerbe}, there are $\gcd(2^m-1, \varphi(p^n))$ solutions. By Lemma \ref{lem:Euler-totient-function}, $\varphi(p^n)=(p-1)p^{n-1}$. By Euler theorem, for every odd prime $p$, we have
%\[
%2^{\varphi(p^{n-1})}\equiv 1\ (\mod p^{n-1}).
%\]
%Hence, if $m=\varphi(p^{n-1})=(p-1)p^{n-2}$, then $\gcd(2^m-1, \varphi(p^n))\geqslant p^{n-1}$.

By using the same idea (as illustrated in Figure \ref{Fig:mod-p2}) in the proof of Proposition \ref{prop:3}, we know that if $x_i^2\equiv x_j\ (\mod p^2)$ and $x_i^2\equiv x'_j\ (\mod p^3)$, then $x'_j\equiv x_j\ (\mod p^2)$. Thus, following the arguments, we get $L(p^3)\leqslant p\cdot L(p^2)$. Note that the multiplication factor is $p$, not $p-1$.

If the largest cycle modulo $p^3$ is less than $pL(p^2)$. That is, the largest cycle doesn't take all of the elements in the area similar in Figure \ref{Fig:mod-p2}. Then, there must be another largest cycle modulo $p^3$. If use curve to describe the cycle, then the curve intersect each vertical line once. They all project onto the "line" $[0,p^2]$. And they are not intersect (i.e., have no common elements). Since $L(p^3)>L(p^2)$, the largest cycle's length must greater than $L(p^2)$. And note that there are $p$ rows($p$ is a prime number), then the number of largest cycles must be $1$. That is $L(p^3)=pL(p^2)$.

For the general cases, the idea is the same. By conduction we complete the proof.
\end{proof}

Finally, we have

\begin{prop}
Suppose $m$ is a composite number, and $m=p_{i_1}^{s_1}p_{i_2}^{s_2}\cdots p_{i_t}^{s_t}$. Here $p_{i_k}$ are all prime numbers. Then we have
\[
L(m)=L(p_{i_1}^{s_1}p_{i_2}^{s_2}\cdots p_{i_t}^{s_t})=\mathrm{lcm}\Bigl(L(p_{i_1}^{s_1}), L(p_{i_2}^{s_2}), \ldots, L(p_{i_t}^{s_t})\Bigr).
\]
\end{prop}
\begin{proof}
Let $m=p_{i}^{s_i}p_{j}^{s_j}$. Consider equation
\[
x^{2^m}\equiv x\ (\mod p_{i}^{s_i}p_{j}^{s_j}),
\]
where $(x,p_i)=(x,p_j)=1$. Then it is equivalent to the equation
\[
x^{2^m-1}\equiv 1\ (\mod p_{i}^{s_i}p_{j}^{s_j}).
\]
By Lemma \ref{lem:bakkerbe}, it has $\gcd(2^m-1, \varphi(p_{i}^{s_i}p_{j}^{s_j}))$ solutions.
\[
\varphi(p_{i}^{s_i}p_{j}^{s_j})=\varphi(p_{i}^{s_i})\varphi(p_{j}^{s_j})=(p_i-1)p_i^{s_i-1}\cdot (p_j-1)p_j^{s_j-1}.
\]
Hence, we have $L(p_{i}^{s_i}p_{j}^{s_j})=\mathrm{lcm}(L(p_{i}^{s_i}),L(p_{j}^{s_j}))$. By conduction we will complete the proof.
\end{proof}

We give a few examples.
\[
\begin{aligned}
 L(15)&=L(3\cdot 5)=1=\lcm(L(3),L(5))=\lcm(1,1),\\
 L(45)&=L(3^2\cdot 5)=2=\lcm(L(3^2),L(5))=\lcm(2,1),\\
L(135)&=L(3^3\cdot 5)=6=\lcm(L(3^3),L(5))=\lcm(6,1),\\
 L(75)&=L(3\cdot 5^2)=4=\lcm(L(3),L(5^2))=\lcm(1,4),\\
L(225)&=L(3^2\cdot 5^2)=4=\lcm(L(3^2),L(5^2))=\lcm(2,4),\\
L(675)&=L(3^3\cdot 5^2)=12=\lcm(L(3^3),L(5^2))=\lcm(6,4),\\
 L(375)&=L(3\cdot 5^3)=20=\lcm(L(3),L(5^3))=\lcm(1,20),\\
L(1125)&=L(3^2\cdot 5^3)=20=\lcm(L(3^2),L(5^3))=\lcm(2,20),\\
L(3375)&=L(3^3\cdot 5^3)=60=\lcm(L(3^3),L(5^3))=\lcm(6,20),\\
\end{aligned}
\]

\begin{remark}
For $L(p^2)$, the situations are complicated. We list the various formulae here.

\noindent(1) The first few primes obey the formula $L(p^2)=(p-1)L(p)$ are:
\[
2,3,5,29,179,293,317,\ldots
\]
(2) The first few primes obey the formula $L(p^2)=\frac{1}{2}(p-1)L(p)$ are:
\[
\begin{aligned}
&7,11,13,17,23,47,59,67,71,83,103,107,131,139,\\
&167,173,191,227,239,263,269,347,\ldots\\
\end{aligned}
\]
(3) The first few primes obey the formula $L(p^2)=\frac{1}{4}(p-1)L(p)$ are:
\[
53,61,97,113,149,193,349,\ldots
\]
(4) The first few primes obey the formula $L(p^2)=\frac{1}{6}(p-1)L(p)$ are:
\[
19,31,37,43,79,199,211,223,229,277,283,\ldots
\]
(5) The first few primes obey the formula $L(p^2)=\frac{1}{7}(p-1)L(p)$ are:
\[
197,\ldots
\]
(6) The first few primes obey the formula $L(p^2)=\frac{1}{8}(p-1)L(p)$ are:
\[
41,89,137,233,281,353,\ldots
\]
(7) The first few primes obey the formula $L(p^2)=\frac{1}{10}(p-1)L(p)$ are:
\[
311,\ldots
\]
(8) The first few primes obey the formula $L(p^2)=\frac{1}{12}(p-1)L(p)$ are:
\[
157,181,\ldots
\]
(9) The first few primes obey the formula $L(p^2)=\frac{1}{18}(p-1)L(p)$ are:
\[
127,271,307,\ldots
\]
(10) The first few primes obey the formula $L(p^2)=\frac{1}{20}(p-1)L(p)$ are:
\[
101,\ldots
\]
(11) The first few primes obey the formula $L(p^2)=\frac{1}{24}(p-1)L(p)$ are:
\[
73,313,\ldots
\]
(12) The first few primes obey the formula $L(p^2)=\frac{1}{40}(p-1)L(p)$ are:
\[
241,\ldots
\]
(13) The first few primes obey the formula $L(p^2)=\frac{1}{48}(p-1)L(p)$ are:
\[
337,\ldots
\]
(14) The first few primes obey the formula $L(p^2)=\frac{1}{50}(p-1)L(p)$ are:
\[
151,\ldots
\]
(15) The first few primes obey the formula $L(p^2)=\frac{1}{54}(p-1)L(p)$ are:
\[
109,163,\ldots
\]
(16) The first few primes obey the formula $L(p^2)=\frac{1}{110}(p-1)L(p)$ are:
\[
331,\ldots
\]
(17) The first few primes obey the formula $L(p^2)=\frac{1}{250}(p-1)L(p)$ are:
\[
251,\ldots
\]

Although we do not sure whether there are some definite laws about $L(p^2)$, some interesting phenomenons should be noted. If write $L(p^2)=\frac{1}{k}(p-1)L(p)$, then the differences between the adjacent primes in the list of each cases are all divisible by the corresponding $k$.
\end{remark}

%------------------------------
%\section{Application to the prime factorization}
%
%The formula for $L(p^2)$ may be used to check a number whether it is a prime.
%
%Suppose the propositions above are true(I think it is true with no doubt). Given any large number $N$, to decomposite into a product of smaller integers

%------------------------------

%------------------------------
% input appendix and/or ack
\bigskip

\noindent{\bf Acknowledgments :}
We would like to express our gratitude to Professor Vilmos Komornik who invite the author to visit the Department of Mathematics in University of Strasbourg. This work is done during the stay of the author in Strasbourg.

We also express our gratitude to Dr. Timothy Foo (a researcher in School of Computing, National University of Singapore, Email: dcsfcht@nus.edu.sg) for sharing his finding. He gives the Theorem \ref{thm:by-Dr.Timothy} with the proof and the Remark \ref{rem:by-Dr.Timothy}.

We thank Professor Vilmos Komornik and Dr. Timothy Foo for discussing this question. And thanks Michel Marcus for giving me the link of OEIS A037178.

%-----------bibliography----------

%------------------------------------------
%Information of authors
\bigskip

\noindent Haifeng Xu\\
School of Mathematical Sciences\\
Yangzhou University\\
Jiangsu China 225002\\
hfxu@yzu.edu.cn\\
\medskip

%%-------------------------


\begin{thebibliography}{10}

\bibitem{Quadratic_residue}
{\em https://en.wikipedia.org/wiki/Quadratic\_residue}.

\bibitem{FermatPrime}
{\em http://mathworld.wolfram.com/FermatPrime.html}

\bibitem{OEIS-A019434}
{\em https://oeis.org/A019434}

\bibitem{OEIS-A037178}
{\em https://oeis.org/A037178}

\bibitem{K-L-S_2001}
{M. Krizek, F. Luca and L. Somer, {\em 17 Lectures on Fermat Numbers – From Number Theory to Geometry}, Springer-Verlag, New York, 2001.}

\bibitem{Tsang}
{\em Cindy Tsang, Fermat Numbers. 2010.}

\bibitem{bakkerbe}
{Benjamin Bakker, PRIMITIVE ROOTS,
{\em http://www2.mathematik.hu-berlin.de/\~bakkerbe/math248/primitive.pdf}}

\bibitem{Blanton-Hurd-McCranie}
{E. L. Blanton, Jr., S. P. Hurd and J. S. McCranie, {\em On a digraph defined by squaring modulo n, Fibonacci Quart.} 30 (Nov. 1992), 322--333.}

\end{thebibliography}
\end{document}